\documentclass[12pt]{article}
\usepackage{amsmath,amsthm,amssymb,amsfonts, graphicx}
\usepackage{authblk}
\usepackage{upgreek}
\usepackage{amsrefs}
\usepackage[unicode,breaklinks=true,colorlinks=true]{hyperref}
\usepackage[left=1.25in,right=1.25in,top=0.7in,bottom=1in]{geometry}
\usepackage{longtable}
\usepackage{epstopdf,tcolorbox}
\usepackage{empheq}
\usepackage{authblk}

\usepackage{titletoc}
\titlecontents{section}[1.5em]{\vspace{-2pt}}%
    {\thecontentslabel\enspace\enspace}%
    {}%
    {\titlerule*[0.5pc]{.}\contentspage}%
\titlecontents{subsection}[3em]{\vspace{-2.5pt}}%
    {\thecontentslabel\enspace\enspace}%
    {}%
    {\titlerule*[0.5pc]{.}\contentspage}%

\theoremstyle{plain}
\newtheorem{theorem}{Theorem}[section]

\newtheorem{lemma}[theorem]{Lemma}
\newtheorem{proposition}[theorem]{Proposition}

\theoremstyle{definition}
\newtheorem{remark}[theorem]{Remark}
\numberwithin{equation}{section}

\newcommand{\R}{\mathbb{R}}

\newcommand{\lec}{\lesssim}
\newcommand{\gec}{\gtrsim}

\newcommand{\p}{\partial}

\newcommand{\e}{\epsilon}
\newcommand{\al}{\alpha}

\renewcommand{\th}{\theta}

\newcommand{\g}{\gamma}
\newcommand{\om}{\omega}

\newcommand{\Si}{\Sigma}
\newcommand{\de}{\delta}

\newcommand{\omb}{\bar{\omega}}
\newcommand{\rb}{\bar{r}}
\newcommand{\zb}{\bar{z}}

\newcommand{\Sb}{\bar{S}}

\newcommand{\bke}[1]{\left ( #1 \right )}
\newcommand{\bkt}[1]{\left [ #1 \right ]}
\newcommand{\bket}[1]{\left \{ #1 \right \}}
\newcommand{\norm}[1]{ \| #1 \|}

\newcommand{\EQ}[1]{\begin{equation} #1 \end{equation}}

\newcommand{\EQS}[1]{\begin{equation}\begin{split} #1 \end{split}\end{equation}}
\newcommand{\EQN}[1]{\begin{equation*}\begin{split} #1 \end{split}\end{equation*}}
\newcommand{\EN}[1]{\begin{enumerate} #1 \end{enumerate}}
\newcommand{\cF}{\mathcal{F}}
\newcommand{\cH}{\mathcal{H}}
\newcommand{\dd}{\,\mathrm{d}}
\usepackage{accents}
\newcommand*{\dt}[1] {\accentset{\mbox{\large\bfseries .}}{#1}}
\newcommand{\pd}{\partial}

\newcommand\Om{\Omega}

\newcommand{\curl} {\mathop{\rm curl}}

\title{
Growth rates for anti-parallel vortex tube Euler flows 
in three and higher dimensions}
\author[1]{Stephen Gustafson}
\author[1,2]{Evan Miller}
\author[1]{Tai-Peng Tsai}
\affil[1]{University of British Columbia, Department of Mathematics}

\affil[1]{gustaf@math.ubc.ca, emiller@math.ubc.ca, ttsai@math.ubc.ca}
\affil[2]{emiller@msri.org}
\date{\vspace{-5ex}}

\begin{document}

\maketitle

\begin{abstract}
We consider axisymmetric, swirl-free solutions of the Euler equations in 
three and higher dimensions, of generalized anti-parallel-vortex-tube-pair-type:
the initial scalar vorticity has a sign 
in the half-space, is
odd under reflection across the plane, is bounded and decays sufficiently  
rapidly at the axis and at spatial infinity.
We prove lower bounds on the growth of such solutions in all dimensions, improving a lower bound proved by Choi and Jeong~\cite{ChoiJeong} in three dimensions. 
\end{abstract}

\setcounter{tocdepth}{2}
\tableofcontents

\section{Introduction}

The incompressible Euler equations
\begin{equation}  \label{euler}
\begin{split}
    \partial_t u+(u\cdot\nabla)u
    +\nabla p &= 0 \\
    \nabla\cdot u &= 0 \\
    u|_{t=0} &= u^0
\end{split}
\end{equation}
are one of the fundamental models in fluid dynamics. 
Here we take the spatial domain to be $\R^d$,
so $(x,t) \in \R^d \times [0,T)$ for some $T \in (0,\infty]$,
$u(x,t) \in \R^d$ is the fluid velocity, 
$p(x,t) \in \R$ is the pressure, and
the initial velocity $u^0$ is a divergence-free vector field on $\R^d$.

\subsection{Axisymmetric, swirl-free Euler flows}

In this paper, we will study axisymmetric, swirl-free solutions of the Euler equations in three dimensions, and their 
generalizations to higher dimensions. That is, 
for $d \geq 3$, we consider solutions of the form
\[
    u(x,t)=u_r(r,z,t)e_r + u_z(r,z,t)e_z,
\]
where
\[
    r = \sqrt{x_1^2+...+x_{d-1}^2}, \;\;
    z =x_d, \;\;
    e_r = \frac{(x_1,...,x_{d-1},0)}{r}, \;\;
    e_z= e_d,
\]
a class which is preserved by the dynamics of the Euler equations~\eqref{euler}.
The divergence-free condition in~\eqref{euler} becomes
\begin{equation} \label{divfree}
  \p_r ( r^{d-2} u_r ) + \p_z ( r^{d-2} u_z ) = 0,
\end{equation}
and so the velocity components may be recovered from a stream function 
$\psi(r,z,t)$ as
\begin{equation}  \label{velfromstream}
  \left[ \begin{array}{c} u_r \\ u_z \end{array} \right] =
  \frac{1}{r^{d-2}} \left[ \begin{array}{c} -\p_z \\ \p_r \end{array} \right]
  \psi .
\end{equation}
The scalar vorticity\footnote{Our sign convention for $\om$ is opposite to that of \cite{ChoiJeong}. It is convenient as we will take $\om \ge 0$ when $z>0$ in \eqref{ass1pt}. 
If we denote $\curl( u_r e_r + u_ze_z)=\om^\th e_\th $ in $\R^3$ as in \cite{MR2429247}, then $\om=- \om^\th$.}
\begin{equation} \label{vorticity}
  \om(r,z,t) = \p_r u_z - \p_z u_r
\end{equation}
then satisfies the evolution equation
\[
    \partial_t\omega + (u\cdot\nabla) \omega
    -(d-2) \frac{u_r}{r} \omega=0, \qquad
    \om|_{t=0} = \om^0 := \p_r (u^0_z) - \p_z (u^0_r) 
\]
which implies that the quantity $\frac{\omega}{r^{d-2}}$ is  transported by the flow:
\begin{equation} \label{transport}
    (\partial_t + u\cdot\nabla) 
    \left[ \frac{\omega}{r^{d-2}} \right] = 0.
\end{equation}
This equation may also be written in the useful form
\begin{equation} \label{transport2}
  \partial_t\omega + \tilde\nabla \cdot (u \omega) = 0 ,  
  \qquad \tilde\nabla : = e_r\pd_r  + e_z\pd_z   ,
\end{equation}
by using~\eqref{divfree}.

The transport equation~\eqref{transport} is central to our
analysis. Of course, to close the system,
the velocity $u$ in~\eqref{transport} must be recovered
from the vorticity via a non-local relation
which realizes~\eqref{divfree} and~\eqref{vorticity}
-- see Section~\ref{BS} for the details. 
One frequently used consequence of the transport~\eqref{transport}
is the preservation of Lebesgue norms:
\begin{equation} \label{normpres}
  \left\| \frac{\om(\cdot,t)}{r^{d-2}} \right\|_p \equiv 
  \left\|  \frac{\om^0}{r^{d-2}} \right\|_p
\end{equation}
for all $1 \leq p \leq \infty$ (assuming initial finiteness).

The other feature of the Euler system~\eqref{euler}
we exploit is the conservation of energy:
\begin{equation} \label{coe}
  E(t) := \frac12\int_{\R^d} |u(x,t)|^2 dx \equiv E(0)
\end{equation}
(assuming $E(0) < \infty$) for as long
as the solution remains sufficiently smooth. 

\subsection{Smooth solutions}

For context and background on axisymmetric, swirl-free Euler flows
in $d=3$, and their extensions to $d \geq 4$,
we refer the reader to~\cite{GMT1}.
Here we will merely state results on the existence of smooth
solutions which suit our present purposes. The existence of global 
smooth solutions in $d=3$ is classical 
\cite{Ladyzhenskaya,Yudovich,Serfati,SaintRaymond,Danchin1,Danchin2},
while the $d \geq 4$ statements are established in~\cite{GMT1},
using standard methods.
There, the existence of global smooth solutions is extended to $d=4$, while for $d \geq 5$ only blow-up criteria are obtained.
We summarize:
\begin{theorem} \label{exist}
Let $d \geq 3$. For initial data 
$u^0 \in H^s\left(\mathbb{R}^d\right)$, $s > 2 + \frac{d}{2}$
which is axisymmetric, swirl-free and divergence-free,
and for which $\frac{\omega^0}{r^{d-2}}\in L^1(\R^d) \cap L^\infty(\R^d)$,
there exists a unique solution of the Euler equations~\eqref{euler} 
$u\in C\left([0,T_{max}),H^s
\left(\mathbb{R}^d\right)\right)
\cap C^1\left([0,T_{max}),H^{s-1}
\left(\mathbb{R}^d\right)\right)$,
on a maximal time interval $[0,T_{max})$,
which is axisymmetric, swirl-free and divergence-free, 
and which conserves energy~\eqref{coe}. Moreover:
\begin{itemize}
\item if $d=3$ or $d=4$, then $T_{max} = \infty$; 
\item if $d \geq 5$ and $T_{max} < \infty$, then \;
$\|\omega(\cdot,t)\|_{L^{1}\left(\mathbb{R}^d\right)}
\geq C_{\om^0} \left(T_{max}-t\right)^{-\frac{2(d-2)}{d-4}}$.
\end{itemize}
\end{theorem}

\subsection{Anti-parallel vortex tubes}

In this paper, we consider initial scalar vorticity
which is non-negative for positive $z$,
\begin{equation}  \label{ass1p}
  \om^0(r,z) \geq 0 \;\; (\om^0 \not\equiv 0) \;\; \mbox{ for } z \geq 0,
\end{equation}
and odd in $z$, 
\begin{equation} \label{ass1o}
  \om^0(r,-z) = -\om^0(r,z),
\end{equation}
a setting which in $d=3$ includes pairs of {\it anti-parallel vortex tubes}.
One motivation for considering these configurations is an extensive
 literature of numerical and experimental work (e.g. \cite{Oshima, LimNickels, ShariffLeonard, ChuWangChangs,
GWRW, ChengLouLim, ChoiJeong} and references therein) 
indicating that such rings should collide, stretch horizontally, and produce rapid vorticity growth.
\begin{remark} \label{preservation}
It is shown in~\cite[Appendix A]{GMT1} that the oddness condition~\eqref{ass1o} is preserved by the Euler flow~\eqref{transport}. In particular, the boundary conditions
\[
  \om(r,0,t) \equiv 0, \qquad u_z(r,0,t) \equiv 0
\]
hold, and the transport equation~\eqref{transport} preserves
the positivity condition~\eqref{ass1p} as well. That is,
\begin{equation} \label{ass1pt}
  \om(r,z,t) \geq 0 \;\; (\om \not\equiv 0) \;\; \mbox{ for } z \geq 0
\end{equation}
and  
\begin{equation} \label{ass1ot}
  \om(r,-z,t) = -\om(r,z,t)
\end{equation}
for $t \in [0,T_{max})$. Moreover, the upper half-space
$\{ z > 0 \}$ is invariant under the flow of such 
velocity fields $u$, and so the norm
preservation~\eqref{normpres} remains true
when restricted to $\{ z > 0 \}$.
\end{remark}
We further assume: $d \geq 3$, and 
$u^0 \in H^s\left(\mathbb{R}^d\right)$, $s > 2 + \frac{d}{2}$
is axisymmetric, swirl-free and divergence-free, with
\begin{equation} \label{ass2}
  \frac{\om^0}{r^{d-2}} \in L^1(\R^d) \cap L^\infty(\R^d) \quad
  \text{and} \quad r \om^0, \; \; z \frac{\om^0}{r^{d-2}} \in L^1(\R^d).
\end{equation}
In particular, Theorem~\ref{exist} applies, producing a solution
satisfying~\eqref{ass1pt}--\eqref{ass1ot}, for which the 
conservation of energy~\eqref{coe} holds.

\subsection{Results}

We formulate our results in terms of the radial moment
\[
  R(t) := \iint\limits_{[0,\infty)^2} r^{d-1} \om(r,z,t) \; dr dz
  \; = \frac{1}{C_d} \int\limits_{z \geq 0}  r^{d-1} 
  \left[ \frac{\om}{r^{d-2}} \right] dx 
   \; = \frac{1}{C_d} \int\limits_{z \geq 0}  r \; \om \; dx \; ,
\]
where $C_d := \mathcal H^{d-2} (\mathbb S_1^{d-2})$.
Note that $0 < R(0) < \infty$ by~\eqref{ass1p} and~\eqref{ass2}.

The first result provides upper bounds for $R(t)$:
\begin{theorem} \label{upper}
Let $u(t)$ be a solution of \eqref{euler} as in Theorem \ref{exist} under the same assumptions on initial data $u^0$ which further satisfies \eqref{ass1p},~\eqref{ass1o}, and~\eqref{ass2}.
For $d=3$: there is $C = C(\om^0)$ such that
\[
  R(t) \leq C(1 + t)^4.
\]
For $d=4$,  there is $C = C(\om^0)$ such that   
\[
  R(t) \leq R(0) e^{Ct}.
\]
For $d \geq 5$, there are $C = C(\om^0)$ and $T = T(\om^0)>0$ so that
for $0\le t <\min(T,T_{max}) $,
\[
  R(t) \leq C (T-t)^{-\frac{2(d-1)}{d-4}}.
\]
\end{theorem}
These upper bounds are straightforward consequences
of the estimate
\begin{equation} \label{ubound}
  \| u_r \|_{L^\infty}^2 \lec \left\| \frac{\om}{r^{d-2}} \right\|_{L^\infty} \| \om \|_{L^1},
\end{equation}
which is proved in~\cite{FengSverak} for $d=3$;
and extended to $d\geq 4$ in~\cite[Section 4]{GMT1}.

Our main result provides lower bounds for $R(t)$,
which improve one from~\cite{ChoiJeong} for $d=3$, 
and generalize it to higher dimensions:
\begin{theorem} \label{lower}

Let $u(t)$ be a solution of \eqref{euler} as in Theorem \ref{exist} under the same assumptions on initial data $u^0$ which further satisfies
\eqref{ass1p},~\eqref{ass1o}, and~\eqref{ass2}.
Let $\e > 0$. 
For $d=3$: there is $C = C(\om^0,\e) > 0$ such that
\[
  R(t) \geq C (1 + t)^{\frac{3}{4}-\e}. 
\]
For $d=4$,  there is $C = C(\om^0,\e)$ such that   
\[
  R(t) \geq C (1 + t)^{\frac{2}{3}-\e}.
\]
For $d \geq 5$, there are $C = C(\om^0,\e, d)$ such that for $t \in [0,T_{max})$,
\[
  R(t) \geq C (1 + t)^{\frac{d}{d^2-2d-2}-\e}.
\]
\end{theorem}

\begin{remark}
The main inspiration for this paper is the recent work of Choi-Jeong~\cite{ChoiJeong}, which gave (among other things) a lower bound $R(t) \geq C(1 + t)^{\frac{2}{15}-\e}$ in  dimension $d=3$.
In broad outline, our argument closely follows theirs,
though key estimates are done in a different way. 
\end{remark}
\begin{remark}
We give here a growth rate lower bound for the moment $R(t)$. For certain initial data this implies a corresponding
a lower bound for certain Lebesgue norms of the vorticity, 
$\| \om(\cdot,t) \|_{L^p}$. This is clear for pure vortex patches,
$\frac{\om}{r^{d-2}} = \chi_{\Om(t)}$ in $z>0$,
but, as shown in~\cite[Corollary 1.2]{ChoiJeong}, also carries over 
to certain smoother vorticities.  
\end{remark}
\begin{remark}
We are considering here relatively smooth solutions 
of the Euler equations, which in particular are 
quite smooth at the axis $r=0$.
Rougher (on the axis) flows may form singularities in finite time in $d=3$ (\cite{Elgindi}).
Moreover, boundary effects may also produce finite time singularities
(\cite{ChenHou}), whereas there is no boundary in our setting. 
The mechanism for vorticity growth here and in~\cite{ChoiJeong} -- vortex stretching toward spatial
infinity -- is quite different from either of those works.
\end{remark}

\subsection{Organization}

Theorem~\ref{upper} is proved, in a slightly generalized form for a 
range of radial moments, in Section~\ref{upperbounds}.
Parallel to the argument in~\cite{ChoiJeong}, the proof of Theorem~\ref{lower}
combines several ingredients: monotonicity of the first
vertical moment is proved in Section~\ref{zmon};
a lower bound for the radial moment $R(t)$
is derived in Section~\ref{rmon}; and
Section~\ref{keupper} contains an upper bound for the 
kinetic energy.
The proof is completed in Section~\ref{lowerproof}.

\subsection{Notation}

Throughout,
\[
  3 \leq d = \mbox{ the spatial dimension.}
\]

We denote
\[
  \Pi = \{ (r,z) \; | \; r > 0, \; z \in \R \}, \qquad
  \Pi_+ =  \{ (r,z) \; | \; r > 0, \; z \geq 0 \}.
\]
Lebesgue norms are taken with respect to the 
usual Lebesgue measure $dx$ on $\R^d$, and normalized so that
for an axisymmetric function $f(r,z)$,
\[
  \| f \|_p^p = \iint\limits_{\Pi} |f(r,z)|^p
  r^{d-2} dr dz = \frac{1}{C_d} \int_{\R^d} |f|^p dx.
\]

As usual, we write
\[
  A \lec B \quad (\mbox{respectively } \; A \gec B)
\]
to indicate there is a constant $C$ (independent of any other 
relevant parameters) such that
\[
  A \leq C B \quad (\mbox{respectively } \; A \geq C B),
\]
and in case the implied constant depends on some relavent parameters $\sigma$, we write
\[
  A \lec_\sigma B \quad (\mbox{respectively } \; A \gec_\sigma B).
\]
Finally,
\[
  A \sim B \quad \mbox{ means } \quad A \lec B \mbox{ and } A \gec B.
\]

\section{Velocity components, moments, and upper bounds} 
\label{upperbounds}

This section concerns upper bounds for the radial moments of the vorticity, and contains the proof
of Theorem~\ref{upper}.

\subsection{Biot-Savart law}  \label{BS}

The estimates in this paper originate from the following expression for the stream function in terms of the scalar vorticity (a detailed derivation can be found in~\cite[Section 4]{GMT1}):
\begin{equation} \label{stream} 
  \psi(r,z,t) = -\frac{1}{2\pi} \iint\limits_{\Pi} \cF(S)  \; (r \rb)^{\frac{d}{2}-1} 
  \; \omb \;,
\end{equation}
where here (and subsequently) we use the shorthand
\[
  \omb = \om(\rb,\zb,t) d\rb d\zb \; .
\]
Here 
\[
  S = \frac{(r-\rb)^2 + (z-\zb)^2}{r \rb},
\]
and for $s > 0$,
\begin{equation} \label{Fdef}
  \cF(s) = \int_0^\pi \frac{\sin^{d-3}(\th) \cos(\th) d \th}
{\left[2(1-\cos\th) + s \right]^{\frac{d}{2}-1}}.
\end{equation}
The velocity components are recovered from the stream function 
via~\eqref{velfromstream},
resulting in the expressions
\EQS{\label{ur}
u_r(r,z,t) =  \frac {1} {\pi r^{d-2}} 
\iint\limits_{\Pi} 
\cF'(S) (z-\zb) (r \rb)^{\frac{d}{2}-2} \; \omb }
and
\EQS{\label{uz}
u_z(r,z,t) = -\frac {1} {2\pi r^{d-2}} 
\iint\limits_{\Pi} \bke{\cF'(S) \p_r S +
\frac{d-2}{2r} \cF(S)}
(r\bar r )^{\frac{d}{2}-1} \; \omb .}

\subsection{Radial moments and the proof of Theorem~\ref{upper}}

By the transport equation~\eqref{transport}
and~\eqref{ass1pt} (see Remark~\ref{preservation}),
\begin{equation} \label{om-def}
  \om := \om(r,z,t) dr dz = 
  \frac{\om(r,z,t)}{r^{d-2}}\,r^{d-2} dr\,dz 
\end{equation}
is a (non-negative) measure on $\Pi_+$ of fixed mass
\begin{equation} \label{prob}
  0 < \iint\limits_{\Pi_+} \om \; = \; \iint\limits_{\Pi_+} \om^0 \; \lec_{\om^0} 1.
\end{equation}
Generalizing slightly, we define the $j$-th moment
of $r$ with respect to this measure,
\[
  R_j (t) = \iint\limits_{\Pi_+}r^j \om \;\; \left( =
  \frac{1}{2} \| r^{j-(d-2)} \om \|_1 \; \right)
\]
(using the symmetry~\eqref{ass1ot} for the last equality) 
if it is finite, so that 
$R(t) = R_{d-1}(t)$ and by~\eqref{prob},
\begin{equation} \label{R_0}
  R_0(t) \equiv R_0(0) =  \iint_{\Pi_+} \om^0 \lec_{\om^0} 1.
\end{equation}
Our assumptions~\eqref{ass1p} and~\eqref{ass2}
imply $0 < R_j(0) < \infty$ for $0 \leq j \leq d-1$.

Compute, using~\eqref{transport2}, 
\EQS{ \label{dtRj}
\dt R_j &=  -\iint\limits_{\Pi_+}r^j  \bkt{\pd_r(u_r\om)+\pd_z(u_z \om)}\,dr\,dz
=\iint\limits_{\Pi_+} jr^{j-1}  u_r\om\,dr\,dz,
}
integrating by parts and using the boundary conditions of 
$\om$ on $\pd \Pi_+$. Thus by~\eqref{ubound},
\[
|\dt R_j | \lec \norm{u_r}_\infty R_{j-1} \lec  \norm{\frac{\om}{r^{d-2}}}_\infty^{\frac12} \norm{\om}_1^{\frac12 }R_{j-1}.
\]
Suppose $j>d-2$. Let $b=\frac{j-(d-2)}{2j} \in (0,\frac12)$ and $q=\frac{(d-2)b}{\frac12-b}=j-(d-2)$. By H\"older, we have
\[
\norm{\om}_1^{\frac12} \le \norm{\frac{\om}{r^{d-2}}}_1^{b}
\cdot\norm{r^q\om}_1^{\frac12-b}
=\norm{\frac{\om}{r^{d-2}}}_1^{b}\cdot R_j^{\frac{d-2}{2j}}
  \lec_{\om^0}  R_j^{\frac{d-2}{2j}}
\]
by~\eqref{normpres}, and
\[
  R_{j-1} \le R_0^{\frac 1{j}} \cdot R_j^{\frac {j-1}{j}}
  = R_0(0)^{\frac{1}{j}} \cdot R_j^{\frac {j-1}{j}}
  \lec_{\om^0} R_j^{\frac {j-1}{j}}
\]
by~\eqref{R_0}. Combining the above, using~\eqref{normpres} again, yields
\EQ{\label{dotRjbound}
  |\dt R_j | \lec_{\om^0} 
  R_j^{\frac {d-2}{2j} }\cdot R_j^{\frac {j-1}{j}} = R_j^{1+\frac {d-4}{2j}}.
}
Integrating the differential inequality~\eqref{dotRjbound} 
immediately yields:
\begin{proposition} 
Suppose $d-2 < j <\infty$ and $R_j(0) < \infty$. Then
\EN{
\item $d=3$: \; $R_j(t) \le C(1+t)^{2j}$ \;\;
for $C = C(\om^0,j)$;
\item $d=4$: \; $R_j(t) \le R_j(0) e^{Ct}$ \;\;
for $C = C(\om^0,j)$;
\item $d\ge 5$: \;
there are $C = C(\om^0,j)$ and $T = T(\om^0,j)>0$ so that
for $t <\min(T,T_{max}) $,
$R_j(t) \le C(T-t)^{-\frac {2j}{d-4}}$.
}
\end{proposition}
The case $j=d-1$ establishes Theorem~\ref{upper}. $\Box$

\section{Lower bounds on vorticity stretching}
\label{GrowthSection}

This section contains the proof of Theorem~\ref{lower}.

\subsection{Monotonicity of the first vertical moment}
\label{zmon}

Following~\cite{ChoiJeong}, we introduce the first vertical moment
\[
  Z(t) = \iint_{\Pi_+} z \; \om, \qquad
  \om := \om(r,z,t)\,drdz \;,
\]
noting that $0< Z(0) < \infty$ by assumptions~\eqref{ass1p} and~\eqref{ass2},
and prove:
\begin{proposition}  \label{zprop}
For all $d \ge 3$,  
\[
  \dt Z(t) < 0
\]
(for $t \in [0,T_{max})$ when $d \geq 5$).
\end{proposition}
\begin{proof}
Compute, using~\eqref{transport2}, integration by parts,
and the boundary conditions on $\p \Pi_+$,
\begin{equation} \label{zdot}
  \dt Z = \iint\limits_{\Pi_+}  z \; \pd_t \om(r,z,t) drdz
  = \iint\limits_{\Pi_+}   u_z \; \om \; .
\end{equation}
By \eqref{uz},
\EQN{
  u_z(r,z,t) = -\frac {1} {2\pi r^{d-2}} \iint\limits_{\Pi} K(r,z;\rb,\zb)
  \; (r\bar r )^{\frac{d-2}{2}} \; \omb \;,
  \qquad \omb := \om(\rb,\zb,t)\, d \rb d \zb \;,
 }
where
\EQN{
K(r,z;\rb,\zb) &=  \cF'(S) \pd_r S + \frac{d-2}{2r} \cF(S)
=  \bke{\frac{2(r-\bar r)}{r\bar r} -\frac{1}{r} S } \cF'(S) +\frac{d-2}{2r} \cF(S)
\\
& = \frac 1{r \bar r} \bkt{2(r-\bar r) \cF'(S) + \bar r \cF^*(S)},
}
and
\[
  \cF^*(s) := \frac{d-2}{2} \cF(s) - s  \cF'(s).
\]
Thus from~\eqref{zdot},
\EQN{
  -\dt Z = \iint\limits_{\Pi_+} \frac {1} {2\pi r^{d-2}} 
  \iint\limits_{\Pi} K(r,z;\rb,\zb) \; (r\bar r )^{\frac{d-2}{2}}
  \; \omb \; \om.
}
Changing the domain of the inner integral from $\Pi$ to $\Pi_+$, using 
the oddness $\om(\bar r,-\bar z) = -\om(\bar r,\bar z)$, yields
\EQN{
-\dt Z&=\iiiint\limits_{(\Pi_+)^2}  \rb^{d-2}
\bkt{K(r,z;\bar r,\bar z)-K(r,z;\bar r,-\bar z)} \,  
\frac { \om \; \bar\om} {2\pi (r\bar r)^{\frac{d-2}{2}}}.}
Symmetrizing in $(r,z) \leftrightarrow (\bar r,\bar z)$, noting that both 
\begin{equation} \label{SSbar}
  S =   \frac{(r-\rb)^2 + (z-\zb)^2}{r \rb} \quad \mbox{ and } \quad 
  \bar S =  \frac{(r-\rb)^2 + (z+\zb)^2}{r \rb}  
\end{equation}
are unchanged, we arrive at
\EQ{\label{Zdotza.sym}
  -\dt Z =\iiiint\limits_{(\Pi_+)^2} 
  \tilde K(r,z;\rb,\zb)\,  \frac { \om \bar \om} {4\pi (r\bar r)^{\frac{d}{2}}},
}
where
\[
\begin{split}
\tilde K &=
\quad  \bar r^{d-2} \bket{2(r-\bar r) \cF'(S) +\bar r \cF^*(S)
- 2(r-\bar r) \cF'(\bar S) -\bar r \cF^*(\bar S)}
\\
&\quad +   r^{d-2} \bket{-2(r-\bar r) \cF'(S) +r \cF^*(S)
+ 2(r-\bar r) \cF'(\bar S) - r \cF^*(\bar S)}
\\
&= \cH(r,\bar r, S) - \cH(r,\bar r, \Sb)
\end{split}
\]
and
\begin{equation}  \label{Hform}
\begin{split}
  \cH(r,\bar r, s)
  &=2 \bar r^{d-2} (r-\bar r) \cF'(s) +\bar r^{d-1} \cF^*(s)
  -2 r^{d-2} (r-\bar r) \cF'(s) + r^{d-1} \cF^*(s) \\
  &=-2(r^{d-2}- \bar r^{d-2}) (r-\bar r) \cF'(s) +
  (r^{d-1} +\bar r^{d-1}) \cF^*(s).
\end{split}
\end{equation}
To prove Proposition~\ref{zprop}, 
by~\eqref{Zdotza.sym}, it suffices to show $\tilde K >0$.
Since $S < \bar S$, by~\eqref{Hform}, it suffices to show both
$- \cF'(s)$ and $\cF^*(s)$ are decreasing in $s$ for all $s>0$. 
Indeed, we have:
\begin{lemma} \label{Flem}
For all $d \geq 3$ and $s > 0$,
\begin{equation}  \label{F1}
  \cF''(s) > 0
\end{equation} 
and
\begin{equation} \label{F2} 
  (\cF^*)'(s) < 0.
\end{equation}
\end{lemma}
Proposition~\ref{zprop} follows, subject to the proof of Lemma~\ref{Flem}. 
\end{proof}
\begin{proof}
Using~\eqref{Fdef}, we compute
\begin{equation}  \label{mon1}
\begin{split}
 \cF'(s) &= -\frac{(d-2)}{2} \int_0^\pi \frac{(\sin \th)^{d-3}\cos \th}{[2(1-\cos \th) + s]^{\frac{d}{2}}}\,d\th \\
 &=  -\frac{(d-2)}{2} \int_0^\frac{\pi}{2} (\sin \th)^{d-3} \cos \th
 \left[  \frac{1}{2(1-\cos \th) + s]^{\frac{d}{2}}} -  
 \frac{1}{2(1+\cos \th) + s]^{\frac{d}{2}}} \right] \; < 0
\end{split}
\end{equation}
by changing variable $\th \mapsto \pi - \th$ on $[\frac{\pi}{2},\pi]$, and
\begin{equation} \label{mon2}
\begin{split}
&  \cF''(s) = \frac{d(d-2)}{4} \int_0^\pi  \frac{(\sin \th)^{d-3}
  \cos \th}{[2(1-\cos \th) + s]^{\frac{d}{2}+1}}\,d\th  \\
  &= \frac{d(d-2)}{4} \int_0^\frac{\pi}{2}    
  (\sin \th)^{d-3} \cos \th \left[  
  \frac{1}{2(1-\cos \th) + s]^{\frac{d}{2}+1}} - 
  \frac{1}{2(1+\cos \th) + s]^{\frac{d}{2}+1}}
  \right] d \th \; > 0,
\end{split}
\end{equation}
which establishes~\eqref{F1}. Next,
\EQN{
(\cF^*)'(s) &= \frac d{ds} \bke{\frac{d-2}{2} \cF(s) - s  \cF'(s)}
= \frac{(d-4)}{2} \cF'(s) - s  \cF''(s),
}
so~\eqref{F2} for $d \geq 4$ follows from~\eqref{mon1} and~\eqref{mon2}. 
The case $d=3$ of~\eqref{F2} is proved in \cite[Lemma 3.3]{ChoiJeong}, 
using some elliptic integral relations.
This completes the proof of Lemma~\ref{Flem}.
\end{proof}

As an immediate consequence of Proposition~\ref{zdot} and~\eqref{ass2}, we have
\begin{equation} \label{zbound}
  0 < Z(t) < Z(0) \lec_{\om^0} 1.
\end{equation}

\subsection{Monotonicity of the $d-1$ radial moment}
\label{rmon}

By \eqref{dtRj} with $j=d-1$,
\EQ{\label{eq9.1}
 \dt R(t) = (d-1) \iint\limits_{\Pi_+} r^{d-2} u_r \; \om \;.
}
Changing variables $\zb \mapsto -\zb$ for $\zb < 0$ in
expression~\eqref{ur}, and using~\eqref{ass1ot}, we have
\EQS{\label{eq9.4}
  u_r(r,z,t) = \frac{1}{\pi r^{d-2}} 
  \iint\limits_{\Pi_+} \bkt{\cF'(S)(z-\bar z) - \cF'(\bar S)(z+\bar z)} 
  \; (r\bar r )^{\frac{d}{2}-2} \; \omb \; .
}
Substituting \eqref{eq9.4} in \eqref{eq9.1}, we get
\EQ{\label{eq9.5}
 \dt R(t) = \frac{d-1}{\pi} \iiiint\limits_{(\Pi_+)^2}
 \bkt{\cF'(S)(z-\bar z) - \cF'(\bar S)(z+\bar z)}  
  \,(r\bar r )^{\frac{d}{2}-2} \,\om \bar \om \; .
}
Noting the anti-symmetry of the factor $\cF'(S)(z-\bar z)$ in 
$z \leftrightarrow \bar z$, its integration in \eqref{eq9.5} vanishes, 
and so  
\EQ{\label{Rdotform}
  \dt R(t) = \frac{d-1}{\pi} 
  \iiiint\limits_{(\Pi_+)^2} \left[ -\cF'(\bar S) \right] (z+\bar z)
  \,(r\bar r )^{\frac{d}2-2} \,\om\bar \om.
}
As an immediate consequence of this expression and~\eqref{mon1},
we have the monotonicity of $R$:
\begin{proposition}  \label{Rprop}
For all $d \ge 3$,  
\[
  \dt R(t) > 0,
\]
(for $t \in [0,T_{max})$ when $d \geq 5$).
\end{proposition}

To get a more precise lower bound for $\dt R$, we first need an elementary estimate for the behaviour of $-\cF'(s)$:
\begin{lemma}
For all $s > 0$,
\EQ{\label{Jest}
 -\cF'(s) \sim \frac 1 {s (1+s)^{\frac d2}}.
}
\end{lemma}
\begin{proof}
From~\eqref{mon1},
\[
  -\cF'(s) \sim \int_0^{\frac \pi 2} 
  (\sin \th)^{d-3} \cos \th \bke{A^{-\frac d2} - B ^{-\frac d2}}\,d\th ,
\]
where
\[
  A = 2(1-\cos \th) + s, \quad B = 2(1+\cos \th) + s.
\]
By a simple calculus Lemma~\ref{th10.1}, for $0<A<B$ and $\al>0$,
\[
A^{-\al} - B ^{-\al} \sim \frac{B-A} {A^{\al} B}.
\]
Thus
\[
-\cF'(s) \sim  \int_0^{\frac \pi 2} \frac{(\sin \th)^{d-3}\cos^2 \th}{ [ 2(1-\cos \th) + s]^{\frac d2}[ 2(1+\cos \th) + s] }\,d\th .
\]
From this expression,
\[
  s \gec 1 \;\; \implies \;\; -\cF'(s) \sim s^{-\frac d2-1},
\]
while  
\EQN{
s \ll 1 \;\; \implies \;\;
-\cF'(s) &\sim \int_0^{\frac \pi 4} \frac{\th^{d-3}}{(\th^2+s)^{\frac d2} }\,d\th
+\int_{\frac \pi 4}^{\frac \pi 2} \cos^2 \th\,d\th 
\\
&\sim \int_0^{\sqrt s} \frac{\th^{d-3}}{s^{\frac d2} }\,d\th
+\int_{\sqrt s}^{\frac \pi 4} \frac{\th^{d-3}}{(\th^2)^{\frac d2} }\,d\th 
+1
 \sim \frac{1}{s},
}
which together give~\eqref{Jest}.
\end{proof}

Returning to~\eqref{Rdotform}, and observing from~\eqref{SSbar} that 
\[
  1 + \Sb \sim 4 + \bar S = 
  \frac1{r \bar r} \bkt{(r+\bar r)^2 + (z+\bar z)^2},
\]
we find
\begin{equation}  \label{Rdotest}
\dt R(t) \sim \iiiint\limits_{(\Pi_+)^2}  \frac {(r\bar r )^{d-1}(z+\bar z) \; \om \omb} {\bkt{(r-\bar r)^2 + (z+\bar z)^2} \bkt{(r+\bar r)^2 + (z+\bar z)^2}^{\frac d2}}.
\end{equation}

\subsection{An upper bound for the kinetic energy}
\label{keupper}

Again following the strategy of~\cite{ChoiJeong}, we prove an upper bound for the kinetic energy~\eqref{coe}
\[
  E = \frac 12\int_{\R^d} |u(x,t)|^2 \dd x 
  = \frac{C_d}2 \iint\limits_\Pi \bkt{(u_r)^2+ (u_z)^2} r^{d-2} \dd r\dd z.
\]
\begin{lemma} 
\begin{equation} \label{KEbound}
E \lec \iiiint\limits_{(\Pi_+)^2}  \frac {z\bar z\,(r\bar r )^{d-1} 
\log \left(2+ r \rb \bkt{(r-\bar r)^2 + (z-\bar z)^2}^{-1} \right)
} {\bkt{(r-\bar r)^2 + (z+ \bar z)^2} \bkt{(r+\bar r)^2 + (z-\bar z)^2}^{\frac d2}} 
   \,\om\bar \om.
\end{equation}
\end{lemma}
\begin{proof}
By~\eqref{velfromstream}, and integrating by parts using $\psi|_{r=0}=\psi|_{z=0}=0$,
\EQN{
E = c \iint\limits_\Pi \bkt{(-\pd_z \psi) u_r+ (\pd_r \psi)u_z} \dd r\dd z
= c \iint\limits_\Pi \psi \bkt{\pd_z  u_r - \pd_r u_z} \dd r\dd z
= -c \iint\limits_\Pi \psi \; \om
}
with $c > 0$. Using~\eqref{stream} here yields,
\[
  E = \frac{c}{2\pi} \iiiint\limits_{\Pi^2} \cF(S)  \,(r\bar r )^{\frac{d}2-1} \, \om \bar \om \; .
\]
Using the oddness of $\om$ and $\bar \om$ in $z$ and $\bar z$,
\EQ{ \label{Eform}
E = \frac{c}{\pi} \iiiint\limits_{(\Pi_+)^2} J_E \; (r\bar r )^{\frac{d}2-1} \, \om \bar \om,
\qquad J_E = \cF(S) -\cF(\bar S).
}
By changing $\th \mapsto \pi - \th$ for $\th \in [\frac{\pi}{2},\pi]$,
\EQN{
J_E&
=\int_0^\pi  \frac{(\sin \th)^{d-3} \cos \th}{[2(1-\cos \th) + S]^{\frac d2 -1}}-\frac{(\sin \th)^{d-3} \cos \th}{[2(1-\cos \th) +\bar S]^{\frac d2-1}}\,d\th
\\
&= \int_0^{\frac\pi 2} \Bigg\{ \frac{1}{[2(1-\cos \th) + S]^{\frac d2-1}}-\frac{1}{[2(1-\cos \th) +\bar S]^{\frac d2-1}}
\\
&\qquad\qquad\qquad   - \frac{1}{[2(1+\cos \th) + S]^{\frac d2-1}}+ \frac{1}{[2(1+\cos \th) +\bar S]^{\frac d2-1}}\Bigg\} (\sin \th)^{d-3} \cos \th\,d\th.
}
By calculus Lemma~\ref{th10.1}, and $\bar S-S=4 \frac {z\bar z}{r \bar r}$,
\[
\begin{split}
J_E
&\sim \frac {z\bar z}{r \bar r}
 \int_0^{\frac\pi 2}  \frac{ (\sin \th)^{d-3} \cos^2 \th\,d\th}{[2(1-\cos \th) + S]^{\frac d2-1}[2(1-\cos \th) +\bar S][2(1+\cos \th) + S]} 
\\
&\sim \frac {z\bar z}{r \bar r(1+S)}
 \int_0^{\frac\pi 2}  \frac{ (\sin \th)^{d-3} \cos^2 \th\,d\th}{[2(1-\cos \th) + S]^{\frac d2-1}[2(1-\cos \th) +\bar S]} .
\end{split}
\]
From here we have an upper bound
\begin{equation} \label{J1}
  J_E \lec \frac {z\bar z}{r \bar r \bar S(1+S)} J_1(S), \qquad J_1(s) =
 \int_0^{\frac\pi 2}  \frac{ (\sin \th)^{d-3} \cos^2 \th\,d\th}{[2(1-\cos \th) + s]^{\frac d2-1}} .
\end{equation}
We have
\[
  s \gec 1 \;\; \implies \;\;  J_1(s) \lec \frac 1{s^{\frac d2-1}},
\]
while
\[
\begin{split}
s \ll 1 \;\; \implies \;\;  
J_1(s) &\sim \int_0^{\sqrt s} \frac{\th^{d-3}}{s^{\frac d2-1} }\,d\th
+\int_{\sqrt s} ^{\frac \pi 4} \frac{\th^{d-3}}{(\th^2)^{\frac d2-1} }\,d\th
+\int_{\frac \pi 4}^{\frac \pi 2} \cos^2 \th\,d\th \\
&\sim  1 + \log \frac1s +1 \sim \log(2+\frac 1 {s}),
\end{split}
\] 
so 
\[
  J_1 \lec \frac 1{(1+S)^{\frac d2-1}} \log(2+\frac 1 {S}),
\]
and returning to~\eqref{J1},
\[
  J_E \lec \frac{(z \zb) \log(2 + S^{-1})}{(r \rb) \Sb (1 + S)^{\frac{d}{2}}}.
\]
Using this in~\eqref{Eform}, noting 
\[
  (r \rb) \Sb = (r-\rb)^2 + (z+\zb)^2, \qquad
  1 + S \sim \frac{(r+\rb)^2 + (z-\zb)^2}{r \rb},
\]  
yields~\eqref{KEbound}.
\end{proof}

\subsection{Lower bounds for the radial moment and the 
proof of Theorem~\ref{lower}}
\label{lowerproof}

For any $0 < \e \ll 1$, split 
\[
  (\Pi_+)^2 = \Si_\e \cup \Si_\e^c,
\]
where
\begin{equation} \label{Sig}
  \Si_\e = \{ (r,z,\rb,\zb) \in (\Pi_+)^2 \; | \; 
(r + \rb)^2 + (z-\zb)^2 \leq \e^2 \left[ 
(r + \rb)^2 + (z + \zb)^2 \right]\}.
\end{equation}
Then from~\eqref{coe} and~\eqref{KEbound}, we have
\[
  1 \sim_{\om^0} E \lec E_\e(t) + E^c_\e(t), \qquad
  E_\e = \iiiint\limits_{\Si_\e} e \,\om \omb \;, \quad
  E_\e^c = \iiiint\limits_{\Si_\e^c} e \,\om \omb \; ,
\]
where 
\[
  e =  \frac {z\bar z\,(r\bar r )^{d-1} \log( 2+ S^{-1} )
} {\bkt{(r-\bar r)^2 + (z+ \bar z)^2} \bkt{(r+\bar r)^2 + (z-\bar z)^2}^{\frac d2}} 
  .
\]

\subsubsection{Estimate in $\Si_\e$}
Since
\[
  (r - \rb)^2 + (z+\zb)^2 \geq (z+\zb)^2 \gec z \zb, \qquad
  (r + \rb)^2 + (z-\zb)^2 \geq (r+\rb)^2 \gec r \rb,
\]
we have
\[
  e \lec (r \rb)^{\frac{d}{2}-1} 
  \log(2 + S^{-1}).
\]
For $(r,z,\rb,\zb) \in \Si_\e$, 
\[
  z \sim \zb \quad \mbox{ and } \quad 
  r + \rb + |z-\zb| \lec \e z,
\]
so for any $0 \leq \mu \leq \frac{d}{2}-1$,
\begin{equation} \label{eest}
  e \lec (r \rb)^{\mu} 
(\e^2 z \zb)^{\frac{d}{2} - 1 - \mu} 
\log(2 + S^{-1}).
\end{equation}
Fix $0 < \de \ll 1$, and take 
\begin{equation} \label{mu-def}
  \mu = \left\{ \begin{array}{cc} 0 & \text{if } d=3 \\
  \frac{d-1}{d-2}(\frac{d}{2}-2 + \de) = \frac{d}{2}-1
  - \frac{d - 2(d-1)\de}{2(d-2)} & \text{if } d \geq 4 \end{array} \right..
\end{equation}
Using~\eqref{eest},
the H\"older inequality with measure $\om \omb$ (recalling~\eqref{prob}),
noting that 
\[
  \frac{\mu}{d-1} + \frac{d}{2}-1-\mu + \de \le 1
\]
(with equality when $d \ge 4$), and~\eqref{zbound}, we get
\begin{align} 
  E_\e &\lec \e^{2( \frac{d}{2} -1 - \mu )} \left( \int_{\Pi_+^2}
  (r \rb)^{d-1} \om \omb \right)^{\frac{\mu}{d-1}}
   \left( \int_{\Pi_+^2} z \zb \; \om \omb \right)^{\frac{d}{2} - 1 - \mu} 
  \| \log(2 + S^{-1}) \|_{L^{\frac{1}{\de}}(\om \omb)} 
  \nonumber \\
  \label{EHolder}
  &\lec  \e^a R^{\frac{2 \mu}{d-1}}
  \| \log(2 + S^{-1}) \|_{L^{\frac{1}{\de}}(\om \omb)}, \qquad a = \left\{ \begin{array}{cc} 
  1 & d=3 \\ \frac{d - 2(d-1)\de}{d-2} & d \geq 4 \end{array} \right..
\end{align}

To estimate the log factor, we have:
\begin{lemma} \label{log}
For $1 \leq p < \infty$,
\[
  \| \log(2 + S^{-1}) \|_{L^p(\om \omb)} \lec_{p,\om^0} \log (2 + R).
\]
\end{lemma}
\begin{proof}
First, by~\eqref{prob},
\[
  \int\limits_{S \gec 1} \log^p(2 + S^{-1}) \; \om \omb
  \lec \int \om \omb \lec_{\om^0} 1.
\]

Next, where $S \ll 1$, we have
\[
  \rb \sim r, \qquad (r-\rb)^2 + (z - \zb)^2 \ll r^2, \qquad
  \log(2 + S^{-1}) \lec \log\left(\frac{r^2}{(r-\rb)^2 + (z - \zb)^2} \right).
\]
Divide the region $S\ll1$ to two subregions depending on the sizes of $r^2$ and $\frac{1}{(r-\rb)^2 + (z - \zb)^2}$. 
Using Jensen's inequality,
\[
\begin{split}
  \int\limits_{S \ll 1, \; r^2 \geq \frac{1}{(r-\rb)^2 + (z - \zb)^2}} & \log^p(2 + S^{-1}) \om \omb
  \lec \int\limits_{(\Pi_+)^2, \; r \geq 2} \log^p (r^2) \; \om \omb
  \lec_{\om^0} \int\limits_{\Pi_+} \log^p (2+r^{d-1}) \; \om \\
  &\le c \int\limits _{\Pi_+} \log^p(2+r^{d-1})  \frac  \om c ,
  \qquad c = \int\limits_{\Pi_+} \om  \\
  &\le c \log^p \int\limits_{\Pi_+} (2+r^{d-1})  \frac \om c
  \lec_{\om^0} \log^p (2 + R) .
\end{split}
\]

Finally, by~\eqref{normpres} and~\eqref{ass2}, we have
\[
  \frac{\om(\rb,\zb,t)}{\rb^{d-2}} \leq  
  \left\| \frac{\om^0}{r^{d-2}} \right\|_{L^\infty} 
  \;\; \implies \;\;
  \om(\rb,\zb,t) \lec_{\om^0 }\rb^{d-2} \lec r^{d-2}.
\]
So for $0 < \rho \leq \rho_0=0.01$, %
\[
\begin{split}
  &\int\limits_{S \ll 1, \; r^2 \leq \frac{1}{(r-\rb)^2 + (z - \zb)^2}} \log^p(2 + S^{-1}) \; \om \omb \\ 
  &\lec
  \int\limits_{(\Pi_+)^2, \; (r-\rb)^2 + (z - \zb)^2 \leq \rho^2}
  \log^p \left(\frac{1}{(r-\rb)^2 + (z - \zb)^2} 
  \right) \om \omb +
  \log^p \left(\frac{1}{\rho^2} \right) \int\limits_{(\Pi_+)^2} 
  \om \omb \\
  & \lec_{\om^0} \iint\limits_{\Pi_+} r^{d-2} \; \om 
  \iint\limits_{(r-\rb)^2 + (z - \zb)^2 \leq \rho^2}
  \log^p \left(\frac{1}{(r-\rb)^2 + (z - \zb)^2} \right)
  d \rb d \zb + \log^p \left(\frac{1}{\rho^2} \right)\\
  &\lec_p \rho^2 \log^p \left(\frac{1}{\rho^2} \right) \iint\limits_{\Pi_+} r^{d-2} \om  +  
  \log^p \left(\frac{1}{\rho^2} \right) \lec_{\om^0}
  \left( \rho^2 R^{\frac{d-2}{d-1}} + 1 \right) \log^p \left(\frac{1}{\rho^2} \right)
\end{split}
\]
where we used H\"older and~\eqref{prob} to get 
$\int_{\Pi_+} r^{d-2} \om \lec_{\om^0} R^{\frac{d-2}{d-1}}$.
Choose $\rho=\min(\rho_0, R^{-\frac{d-2}{2(d-1)}})$,
and add together the three contributions above to finish the proof.
\end{proof}

Return now to~\eqref{EHolder} and use Lemma~\ref{log}:
\[
  E_\e \lec_{\de,\om^0} \e^{a} R^{\frac{2\mu}{d-1}} \log (2 + R).
\]
So we may choose $\e = \e(t)$ small enough so that
\begin{equation} \label{epschoice}
  \e^{a} \sim_{\de,\om^0} \frac{1}{
  R^{\frac{2 \mu}{d-1}} \log(2 + R)}
  \;\; \implies \;\;
  E_\e(t) \leq \frac{1}{2} E
  \;\; \implies \;\;
  E \lec E_\e^c(t).
\end{equation}

\subsubsection{Estimates in $\Si_\e^c$, $d=3$}

In $\Sigma_\e^c$, by~\eqref{Sig},
\begin{equation}  \label{ebound}
  \e^d e \lec  \frac{z \zb (r \rb)^{d-1}  \log(2 + S^{-1})}{[(r-\rb)^2 + (z + \zb)^2]
[(r+\rb)^2 + (z + \zb)^2]^{\frac{d}{2}}} .
\end{equation}

Now specialize to $d=3$ (we treat $d \geq 4$ in the next subsection).
Denote the integrand of the upper bound of $\dt R$ in \eqref{Rdotest} by
\[
K= \frac{ (z + \zb) ( r \rb) ^{2}}
  {[(r-\rb)^2 + (z + \zb)^2]
[(r+\rb)^2 + (z + \zb)^2]^{\frac{3}{2}}}.
\]
Using 
\[
  z \zb \leq (z + \zb)^2 \leq (r - \rb)^2 + (z + \zb)^2 
  \quad \mbox{ and } \quad z \zb + r \rb \lec (z + \zb)^2 + (r + \rb)^2
\]
in~\eqref{ebound}, we find, for any $0< \nu < 1$,
\EQN{
   \e^3 e &\lec (z \zb)^{\frac\nu 2}\bkt{ \frac{(r \rb)^2} { [(z + \zb)^2 + (r + \rb)^2]^{\frac32}}}^{1-\nu} K^\nu \log(2 + S^{-1})
\\
&\lec    (r \rb)^{\frac12-\frac\nu 2+m }   (z \zb)^{\frac\nu 2-m}
   K^\nu \log(2 + S^{-1}),
}
for any
\begin{equation} \label{mcond3}
0 \le m \le \frac{3}{2}(1-\nu), \quad m <\frac \nu 2
\end{equation}
(the last condition ensures a positive exponent).
Then by H\"older and~\eqref{Rdotest},
\[
  \e^3 E_\e^c \lec 
  \left( \int_{\Pi_+^2} (r \rb)^{\frac{\frac12-\frac\nu 2+m } q} \om \omb \right)^q
  \left( \int_{\Pi_+^2} (z \zb)^{\frac{\frac\nu 2-m}{p}} \om \omb \right)^p
  ( \dt R )^\nu 
  \| \log(2 + S^{-1}) \|_{L^\frac{1}{\de}(\om \omb)}
\]
where $ q + p +\nu + \de = 1$.
Now we make specific exponent choices:
\begin{equation} \label{mcond4}
q=\frac12 \bke{\frac12-\frac\nu 2+m },\quad
p = \frac\nu 2-m,\quad
m = \frac 52\nu -\frac32+2\de,
\end{equation}
yielding
\[
\begin{split}
  \e^3 E_\e^c &\lec 
  \left( \int_{\Pi_+^2} (r \rb)^{2} \om \omb \right)^q
  \left( \int_{\Pi_+^2} (z \zb) \om \omb \right)^p
  ( \dt R) ^\nu 
  \| \log(2 + S^{-1}) \|_{L^\frac{1}{\de}(\om \omb)} \\
  &\lec_{\de,\om^0} R^{2q} Z^{2p} (\dt R)^\nu \log(2 + R) 
  \lec_{\om^0} R^{2q} (\dt R)^\nu \log(2+R) 
\end{split}
\]
using Lemma~\ref{log}, and~\eqref{zbound}.
Then from~\eqref{coe} and~\eqref{epschoice},
\[
\begin{split}
  1 &\lec_{\om^0} E \lec E_\e^c \lec_{\de,\om^0} \e^{-3} R^{2q} 
  (\dt R)^\nu \log(2+R) \\
  &\lec_{\de,\om^0} R^{2q} \log^{4}(2+R) (\dt R)^\nu
  = \left[ R^{\frac{2q}{\nu}} \log^{\frac{4}{\nu}}(2+R) \dt R \right]^\nu .
\end{split}
\]
It follows that for any $\eta> \frac{2q}{\nu}$, 
\[
   1 \lec_{ \eta,\om^0} R^{ \eta} \dt R \; \sim \; 
  \frac{d}{dt} \left( R^{ \eta + 1} \right).
\]
Integrating this differential inequality yields
\[
  R(t) \gec_{b,\om^0} (1+t)^b, \quad \mbox{ for any } \; 
  b^{-1}> \frac{2q}{\nu}+1 .%
\]
To minimize $\frac{2q}{\nu}= \frac {1+2m}{2\nu}-\frac 12$,
subject to \eqref{mcond3} and  \eqref{mcond4},  
we choose
\[
\nu = \frac 35,\quad 0<m=2\de\ll 1, \quad q=\frac1{10}+\de,\quad p=\frac3{10}-2\de,
\quad \frac{2q}{\nu}+1 = \frac 43+2\de.
\]
By taking $\de>0$ arbitrarily small, we get $R(t) \gec_{b,\om^0} (1+t)^b$ for any $ b^{-1}>\frac 43$.
which completes the proof of the $d=3$ case of Theorem~\ref{lower}.

\subsubsection{Estimates in $\Si_\e^c$, $d \geq 4$}

Now assume $d \geq 4$. Using 
\[
  z \zb \leq (z + \zb)^2 \leq (r - \rb)^2 + (z + \zb)^2 \;\;
  \mbox{ and } \;\; r \rb, \; (z+\zb)^2 \leq (z + \zb)^2 + (r + \rb)^2
\]
in~\eqref{ebound}, we find, for any $0 \leq \nu \leq \frac{d}{d+1}$,
\[
   \e^d e \lec (r \rb)^{\frac{1}{2}(d-2 -(d-3)\nu)} 
   \left[ \frac{ (z + \zb) ( r \rb) ^{(d-1)}}
  {[(r-\rb)^2 + (z + \zb)^2]
[(r+\rb)^2 + (z + \zb)^2]^{\frac{d}{2}}} \right]^\nu
\log(2 + S^{-1})
\]
We choose
\[
  \nu = \frac{d-2(d-1)\de}{d+1} \quad \mbox{ and } \quad
  q = 1 - \nu - \de = \frac{1 + (d-3)\de}{d+1},
\]
so that by H\"older's inequality,
\begin{multline*}
  \e^d E_\e^c \lec \left( \int_{\Pi_+^2} (r \rb)^{d-1} \om \omb \right)^q
  \left( \int_{\Pi_+^2}   \frac{ (z + \zb) ( r \rb) ^{(d-1)} \om \omb}
  {[(r-\rb)^2 + (z + \zb)^2] [(r+\rb)^2 + (z + \zb)^2]^{\frac{d}{2}}}
  \right)^\nu \cdot\\
  \cdot \| \log(2 + S^{-1}) \|_{L^\frac{1}{\de}(\om \omb)},
\end{multline*}
and then by~\eqref{Rdotest} and Lemma~\ref{log}
\[
  E_\e^c \lec_\de \e^{-d} 
  R^{\frac{2 + 2(d-3)\de}{d+1}}
  ( \dt R)^\nu \log(2 + R)  \lec
  \left[ \e^{-\frac{d}{\nu} } R^\g \log^{\frac{1}{\nu}}(2+R) \dt R \right]^\nu
\]
with 
\[
  \g = \frac{2 + 2(d-3)\de}{d - 2(d-1) \de} = \frac{2}{d} + O(\de).
\]
Then from~\eqref{coe} and~\eqref{epschoice},
\[
  1 \lec_{\om^0} E^{\frac{1}{\nu}} \lec (E_\e^c)^{\frac{1}{\nu}} \lec_\de \e^{-\frac{d}{\nu}} R^{\g} 
  \log^{\frac{1}{\nu}}(2+R) \dt R 
  \lec_{\de,\om^0} R^{\eta} 
  \log^{\frac{a + d}{a\nu}}(2 + R) \dt R
\]
with 
\[  
  \eta = \g + \frac{2 d \mu}{(d-1)a\nu}
  = \frac{2 + (d-4)(d+1)}{d} + O(\de).
\]
It follows that
\[
 \mbox{ for any } \; \tilde \eta >
 \frac{2 + (d-4)(d+1)}{d}, \qquad
 1 \lec_{\tilde \eta, \om^0} R^{\tilde \eta} \dt R
 \sim \frac{d}{dt} \left( R^{\tilde \eta + 1} \right).
\]
Integrating this differential inequality yields
\[
  R(t) \gec_{b,\om^0} (1+t)^b, \;\; \mbox{ for any } \; 
  b < \frac{d}{d^2-2d-2},
\]
which completes the proof of the $d \geq 4$ cases of Theorem~\ref{lower}.
$\Box$

\section{Appendix: a calculus lemma}

\begin{lemma}
\label{th10.1}
Let $f(x)=x^{-\al}$, $\al>0$. For $x,y,z>0$, we have
\EQ{\label{th10.1a}
f(x) - f(x+y) \sim f(x) \frac y {x+y},
}
\EQ{\label{th10.1b}
f(x) - f(x+y) -f(x+z) + f(x+y+z) \sim f(x) \frac y {x+y}\frac z {x+z}.
}
\end{lemma}

\begin{proof}
Fix $L>0$ such that $f(x+y) <\frac 14 f(x) $ if $y>Lx$, e.g., $L=4^{\frac 1\al}-1$.
If $y >Lx$, then
\[
f(x) > f(x) - f(x+y) > \frac 34 f(x)
\]
while $\frac y {x+y} \sim 1$. If $x<y<Lx$, then for some $\th \in [0,1]$,
\[
 f(x) - f(x+y) = -f'(x+\th y) y \sim f(x) \frac y{x+y}.
 \]
The above shows \eqref{th10.1a}. For \eqref{th10.1b}, denote its left side as $g(x,y,z)$. If $x<y<Lx$, then for some $\th \in [0,1]$ we have
\EQN{
g(x,y,z) &= g(x,y,z)-g(x,0,z) = y \pd_y g(x,\th y,z)
\\
&=y[ - f'(x+\th y) + f'(x+\th y+z) ].
}
Since $f'(x) = -\al x^{-\al-1}$ has a similar form as $f(x)$, by  \eqref{th10.1a},
\[
g(x,y,z) \sim y |f'(x+\th y) |\frac {z}{x+\th y+z}\sim y \frac {f(x)}{x+y} \frac {z}{x+z}
\]
 If $x<z<Lx$, we have the same estimate by symmetry. Finally, if both $y,z>Lx$, then
 \[
 f(x)+ \frac 14 f(x)  > g(x,y,z) \ge f(x) - \frac 14 f(x) - \frac 14 f(x) 
 \]
 while $\frac y {x+y}\frac z {x+z} \sim 1$. The above shows \eqref{th10.1b}. 
\end{proof}

\section*{Acknowledgements}
The authors were partially supported by NSERC under the grants
RGPIN-2018-03847 and RGPIN-2018-04137.
EM was also supported by the Pacific Institute for the Mathematical Sciences as a PIMS postdoctoral fellow.

\addcontentsline{toc}{section}{\protect\numberline{}{References}}


\begin{bibdiv}
\begin{biblist}

\bib{MR2429247}{article}{
      author={Chen, Chiun-Chuan},
      author={Strain, Robert~M.},
      author={Yau, Horng-Tzer},
      author={Tsai, Tai-Peng},
       title={Lower bound on the blow-up rate of the axisymmetric
  {N}avier-{S}tokes equations},
        date={2008},
        ISSN={1073-7928},
     journal={Int. Math. Res. Not. IMRN},
      number={9},
       pages={Art. ID rnn016, 31},
         url={https://doi-org.ezproxy.library.ubc.ca/10.1093/imrn/rnn016},
      review={\MR{2429247}},
}

\bib{ChenHou}{article}{
      author={Chen, Jiajie},
      author={Hou, Thomas},
       title={Stable nearly self-similar blowup of the 2d boussinesq and 3d
  euler equations with smooth data},
        date={2022},
        note={arXiv:2210.07191},
}

\bib{ChengLouLim}{article}{
      author={Cheng, M.},
      author={Lou, J.},
      author={Lim, T.T.},
       title={Numerical simulation of head-on collision of two coaxial vortex
  rings},
        date={2018},
     journal={Fluid Dyn. Res.},
      volume={50},
      number={24},
       pages={065513},
}

\bib{ChoiJeong}{article}{
      author={Choi, Kyudong},
      author={Jeong, In-Jee},
       title={On vortex stretching for anti-parallel axisymmetric flows},
        date={2021},
        note={arXiv:2110.09079},
}

\bib{ChuWangChangs}{article}{
      author={Chu, C.},
      author={Wang, C.},
      author={Chang, C},
      author={Chang, R},
      author={Chang, W},
       title={Head-on collision of two coaxial vortex rings: Experiment and
  computation},
        date={1995},
     journal={J. Fluid Mech.},
      volume={296},
       pages={39\ndash 71},
}

\bib{Danchin1}{article}{
      author={Danchin, Rapha\"{e}l},
       title={Axisymmetric incompressible flows with bounded vorticity},
        date={2007},
        ISSN={0042-1316},
     journal={Uspekhi Mat. Nauk},
      volume={62},
      number={3(375)},
       pages={73\ndash 94},
         url={https://doi.org/10.1070/RM2007v062n03ABEH004412},
      review={\MR{2355419}},
}

\bib{Danchin2}{article}{
      author={Danchin, Rapha\"{e}l},
       title={On perfect fluids with bounded vorticity},
        date={2007},
        ISSN={1631-073X},
     journal={C. R. Math. Acad. Sci. Paris},
      volume={345},
      number={7},
       pages={391\ndash 394},
         url={https://doi.org/10.1016/j.crma.2007.09.002},
      review={\MR{2361504}},
}

\bib{Elgindi}{article}{
      author={Elgindi, Tarek},
       title={Finite-time singularity formation for {$C^{1,\alpha}$} solutions
  to the incompressible {E}uler equations on {$\Bbb R^3$}},
        date={2021},
        ISSN={0003-486X},
     journal={Ann. of Math. (2)},
      volume={194},
      number={3},
       pages={647\ndash 727},
         url={https://doi.org/10.4007/annals.2021.194.3.2},
      review={\MR{4334974}},
}

\bib{FengSverak}{article}{
      author={Feng, Hao},
      author={\v{S}ver\'{a}k, Vladim\'{\i}r},
       title={On the {C}auchy problem for axi-symmetric vortex rings},
        date={2015},
        ISSN={0003-9527},
     journal={Arch. Ration. Mech. Anal.},
      volume={215},
      number={1},
       pages={89\ndash 123},
         url={https://doi.org/10.1007/s00205-014-0775-4},
      review={\MR{3296145}},
}

\bib{GWRW}{article}{
      author={Guan, Hui},
      author={Wei, Zhi-Jun},
      author={Rasolkova, Elizabeth~Rumenova},
      author={Wu, Chui-Jie},
       title={Numerical simulations of two coaxial vortex rings head-on
  collision},
        date={2016},
     journal={Adv. Appl. Math. Mech},
      volume={8},
      number={4},
       pages={616\ndash 647},
}

\bib{GMT1}{article}{
      author={Gustafson, Stephen},
      author={Miller, Evan},
      author={Tsai, Tai-Peng},
       title={Regularity of axisymmetric, swirl-free {E}uler flows in higher
  dimensions},
        date={2023},
     journal={preprint},
}

\bib{Ladyzhenskaya}{article}{
      author={Lady\v{z}enskaya, O.~A.},
       title={Unique global solvability of the three-dimensional {C}auchy
  problem for the {N}avier-{S}tokes equations in the presence of axial
  symmetry},
        date={1968},
     journal={Zap. Nau\v{c}n. Sem. Leningrad. Otdel. Mat. Inst. Steklov.
  (LOMI)},
      volume={7},
       pages={155\ndash 177},
      review={\MR{0241833}},
}

\bib{LimNickels}{article}{
      author={Lim, T.T.},
      author={Nickels, T.B.},
       title={Instability and reconnection in the head-on collision of two
  vortex rings},
        date={1992},
     journal={Nature},
      volume={357},
       pages={225\ndash 227},
}

\bib{Oshima}{article}{
      author={Oshima, Y.},
       title={Head-on collision of two vortex rings},
        date={1978},
     journal={J. Phys. Soc. Japan},
      volume={44},
      number={44},
       pages={328\ndash 331},
}

\bib{SaintRaymond}{article}{
      author={Saint~Raymond, X.},
       title={Remarks on axisymmetric solutions of the incompressible {E}uler
  system},
        date={1994},
        ISSN={0360-5302},
     journal={Comm. Partial Differential Equations},
      volume={19},
      number={1-2},
       pages={321\ndash 334},
         url={https://doi.org/10.1080/03605309408821018},
      review={\MR{1257007}},
}

\bib{Serfati}{article}{
      author={Serfati, Philippe},
       title={R\'{e}gularit\'{e} stratifi\'{e}e et \'{e}quation d'{E}uler
  {$3$}{D} \`a temps grand},
        date={1994},
        ISSN={0764-4442},
     journal={C. R. Acad. Sci. Paris S\'{e}r. I Math.},
      volume={318},
      number={10},
       pages={925\ndash 928},
      review={\MR{1278153}},
}

\bib{ShariffLeonard}{article}{
      author={Shariff, K.},
      author={Leonard, A},
       title={Vortex rings},
        date={1992},
     journal={Ann. Rev. Fluid Mech.},
      volume={24},
      number={24},
       pages={235\ndash 279},
}

\bib{Yudovich}{article}{
      author={Ukhovskii, M.~R.},
      author={Yudovich, V.~I.},
       title={Axially symmetric flows of ideal and viscous fluids filling the
  whole space},
        date={1968},
        ISSN={0021-8928},
     journal={J. Appl. Math. Mech.},
      volume={32},
       pages={52\ndash 61},
         url={https://doi.org/10.1016/0021-8928(68)90147-0},
      review={\MR{0239293}},
}

\end{biblist}
\end{bibdiv}


\end{document}